\documentclass[journal,twoside,web]{ieeecolor}

\usepackage{generic}
\usepackage{hyperref}
\hypersetup{hidelinks=true}
\usepackage{cite}
\usepackage{amsmath,amssymb,amsfonts}
\usepackage[linesnumbered, ruled,vlined]{algorithm2e}
\usepackage{graphicx}
\usepackage{textcomp}
\usepackage{enumerate}
\newtheorem{lemma}{Lemma}

\newtheorem{corollary}{Corollary}

\usepackage{booktabs}
\usepackage{multirow}
\usepackage{orcidlink}
\usepackage{mathtools}

\newcommand{\vct}[1]{\mathbf{#1}}

\def\BibTeX{{\rm B\kern-.05em{\sc i\kern-.025em b}\kern-.08em
    T\kern-.1667em\lower.7ex\hbox{E}\kern-.125emX}}
\begin{document}
\title{Data-Driven Static Output-Feedback Control for Multi-Agent Systems}
\author{Paul Serna-Torre \orcidlink{0000-0003-0578-7933}, Adam Sedlak \orcidlink{0000-0002-9487-3210}, and Patricia Hidalgo-Gonzalez \orcidlink{0000-0003-3635-8285}
\thanks{This paragraph of the first footnote will contain the date on 
which you submitted your paper for review. 
 }
\thanks{P.S. is supported by the Center for Energy Research and a scholarship from the Peru Ministry of Education. A.S is supported by Center for Energy Research and the National Science Foundation Graduate Research Fellowship Program under Grant No DGE-2545911. Any opinions, findings, and conclusions or recommendations expressed in this material are those of the authors and do not necessarily reflect the views of the National Science Foundation.}
\thanks{P.S, A.S. and P.H are with the Center for Energy Research and the Department of Mechanical and Aerospace Engineering at University of California, San Diego, California, USA.
CO 80523 USA (e-mail: \{psernatorre\}\{asedlak\}\{phidalgogonzalez\}@ucsd.edu).}
}

\maketitle

\begin{abstract}
This work proposes a data-driven framework that synthesizes output-feedback control for linear time-invariant multi-agent systems (MAS). We first develop two algorithms that use offline data samples of state and input trajectories to determine a Nash equilibrium (NE) for MAS. Next, we formulate a semidefinite program that, given the NE solution, computes the static output-feedback controllers for the agents in MAS. Compared to prior work, virtues of our framework include weaker controllability assumptions, the absence of online information and data-driven estimation of states. Next,  we use the proposed framework to design control schemes for voltage and frequency regulation for nonlinear multi-nodal electrical networks with grid-forming and grid-following inverters. Thus, compared to prior work that implements their data-driven frameworks in small-scale numerical examples, we validate the performance of our data-driven framework in high-fidelity nonlinear models of MAS. This implementation in realistic MAS provides insights that small-scale numerical examples may fail to capture. Results show successful grid stabilization and provision of regulation services from the inverters in the grid using our proposed framework. It can attain up to 78\% less overshoot and 30\% faster settling time than traditional strategies such as proportional-integral control.    
\end{abstract}


\section{Introduction}
\label{sec:introduction}
\IEEEPARstart{D}{}ifferential games can serve to solve critical control problems in real-world multi-agent systems (MAS) \cite[Chapter 8]{Engwerda_LQbook}. For instance, in a multi-robot system, differential games can be used to synthesize a control strategy for agents to reach a goal while maintaining sufficient distance from one another\cite{fridovichkeil2020games}. In a power grid, inverters can be controlled using differential games to inject a requested power profile in benefit of the grid \cite{noncooperative_paulpatricia}. Although model-based frameworks are quite versatile, their practical use can be severely limited in real systems. For instance, network interactions between agents may contain dynamics and cannot be modeled as static graphs. Alternatively, the agents in the MAS can also be heterogeneous and complex, meaning they cannot be modeled by simple single-state models. Additionally, the access to the parameters of the agents and the network may not be accessible due to confidentiality. 

To overcome the mentioned challenges for model-based control, several works have proposed the design of control frameworks using data directly \cite{data_driven_tutorial}. This body of work builds upon a fundamental lemma proposed by \cite{willems_note_excitation}. This lemma illustrates that the data obtained from exciting an unknown model is rich enough to span the full behavior of the model if the input signal is persistently exciting \cite{alsalti_excitation}. 
Assuming persistency of excitation, the authors in \cite{dasilva_lqr_markov} use data samples of input and state trajectories to estimate Markov parameters and an extended observability matrix, which serve to synthesize a linear quadratic regulator (LQR) strategy. The work in \cite{depersis_lqr_datadriven} 
synthesizes the LQR controller through semidefinite programs (SDPs) that use offline data samples of state and input trajectories. The work in \cite{dai2023lqrnoise, depersis_robust_lqr} extends these SDPs to incorporate bounded noise measurements. In \cite{vanwaarde_suboptimal_lqr_data}, the authors propose a SDP that requires less data, although the resulting control strategies are suboptimal relative to the LQR strategy. More recently, the work in \cite{dai_output_datadriven, OLIVEIRA2025101260} proposes sufficient stabilization conditions that can be written as a SDP to ultimately synthesize a static output-feedback strategy using data directly.

Although the prior work in \cite{dasilva_lqr_markov, depersis_lqr_datadriven, depersis_robust_lqr, dai_output_datadriven, vanwaarde_suboptimal_lqr_data, dai2023lqrnoise, OLIVEIRA2025101260} lays an important foundation for data-driven control, these methods may not be readily applicable to MAS. In particular these frameworks assume the existence of a system-wide objective function. In contrast, in MAS, multiple objective functions co-exist. Moreover, each agent can be self-interested as it seeks to optimize its own objective function. \textcolor{black}{In efforts to bridge this gap, recently the authors in \cite{nortmann_nash_iterative, lopez2026databasednashequilibrialq} propose algorithms that determine the Nash equilibrium (NE) in LTI MAS using data directly.} These algorithms use offline data and synthesize control strategies for the agents to attain a NE. However, the mentioned works assume strategies under state-feedback structure, thus requiring full state information. In real world systems this requirement could be challenging to meet due to limited sensors and real-time communication infrastructure. 

Alternatively, several works have proposed data-driven frameworks that synthesize output-feedback strategies for MAS. The frameworks in \cite{kedi2025output, RIZVI2020output} propose using an observer to bypass the need for state feedback. That observer is embedded in model-free policy iteration schemes that compute the output-feedback control. The authors in \cite{ODEKUNLE2020108672, zhang2026output_qlearning} propose algorithms that use reinforcement learning and online data to ultimately find the Nash-seeking output-feedback strategies. Notwithstanding the mentioned developments, it is still an open question whether there exist a straight-forward approach to synthesize output-feedback strategies. For instance, an SDP, that off-the-shelf optimizers can solve, without needing online data samples and or iterative learning.

Furthermore, most of surveyed works proposing data-driven control synthesis, including the mentioned articles in \cite{depersis_lqr_datadriven, depersis_robust_lqr, dai_output_datadriven, vanwaarde_suboptimal_lqr_data, dai2023lqrnoise, OLIVEIRA2025101260, nortmann_nash_iterative, lopez2026databasednashequilibrialq, kedi2025output, RIZVI2020output, ODEKUNLE2020108672, zhang2026output_qlearning}, implement their frameworks in small-scale numerical examples. More specifically, these examples assume given numerical state-space models from which next data samples are extracted to test their frameworks. Yet, as highlighted by \cite{zeng2024noisesensitivity}, simplistic numerical examples may overlook real-world challenges, such as the impact of measurement noise on data-driven control methods. Additionally, real-world MAS are typically inherently non-linear. Thus, it is also necessary to experimentally validate the performance of a control framework, designed using its LTI MAS simplification, in high-fidelity non-linear MAS. Experimental validation in realistic settings enables us to better understand the barriers to real-world implementations of the proposed data-driven control framework. Moreover, such validation exercise can serve as an important step to guiding future research and ensure that data-driven control frameworks do not remain only theoretical in nature.

\section{Contributions}
The major contributions of this work are as follows:
\begin{enumerate}
    \item 
    We design two algorithms that determine the Nash equilibrium using offline data samples of state and input trajectories. Algorithm \ref{algo1}, although it can be similar to the work in \cite{nortmann_nash_iterative, lopez2026databasednashequilibrialq}, enables output-feedback control synthesis. We also introduce Algorithm \ref{algo2} which requires weaker controllability/stabilizability assumptions than both Algorithm \ref{algo1} and algorithms in prior work \cite{nortmann_nash_iterative, lopez2026databasednashequilibrialq}. 
    \item We propose a SDP that utilizes the NE solution from either Algorithm~\ref{algo1} or Algorithm~\ref{algo2} to directly synthesize the static output-feedback control strategies in LTI MAS. A virtue of this SDP is that it does not require either the knowledge of the input and state matrices nor any data sample. Additionally, it is not part of iterative schemes. Notice that, our proposed data-driven framework is comprised of a two-step process: first, the computation of the NE through Contribution (1), and second, the computation of the static output-feedback control through Contribution (2). Compared to prior work in data-driven output-feedback control for MAS, our framework avoids the use of state estimators, dynamic gains, and online data samples. 
    \item We assess \textcolor{black}{the scalability of our proposed framework in a realistic multi-agent networked system. To this end, we develop} control schemes for frequency and voltage regulation in \textcolor{black}{multi-nodal inverter-dominated power grids}. We implement these control schemes in electromagnetic transient (EMT) models, which capture the inherently non-linear dynamics of power systems. We showcase the performance of our framework against traditional control such as proportional and integral control. Prior work, such as  \cite{depersis_lqr_datadriven, depersis_robust_lqr, dai_output_datadriven, vanwaarde_suboptimal_lqr_data, dai2023lqrnoise, OLIVEIRA2025101260, nortmann_nash_iterative, lopez2026databasednashequilibrialq, kedi2025output, RIZVI2020output, ODEKUNLE2020108672, zhang2026output_qlearning}, tests their frameworks in small-scale numerical cases or does not implement their methods in \textcolor{black}{realistic multi-agent networked systems}. To the best of our knowledge, this work is the first of its kind in showcasing the effectiveness of a data-driven multi-agent output-feedback control and its implementation in the \textcolor{black}{EMT model of electrical power systems}.       
    \item We provide an \textcolor{black}{open-source} toolbox that implements the proposed Algorithms~\ref{algo1} and \ref{algo2} and SDP from Contributions (1) and (2). While existing software can synthesize data-driven single-agent control; extensions to MAS are not as readily available. Moreover, 
    \textcolor{black}{although several frameworks for data-driven control in MAS have been proposed, \cite{dasilva_lqr_markov, depersis_robust_lqr, dai_output_datadriven, vanwaarde_suboptimal_lqr_data, dai2023lqrnoise, OLIVEIRA2025101260, nortmann_nash_iterative, lopez2026databasednashequilibrialq, kedi2025output, RIZVI2020output, ODEKUNLE2020108672, zhang2026output_qlearning}, their computational implementations are not readily available or provided. We also aim to bridge such gap by implementing our proposed framework as a toolbox which allows other researchers and practitioners to readily leverage our framework.}
\end{enumerate}

\section{Problem Formulation}
\label{sec_problem_statement}


\subsection{Problem 1: Nash equilibrium solution}
\label{sec_problem1}
Consider a MAS that consists of a set of agents $\mathcal{C} =\{1, ..., |\mathcal{C}| \}$, where $|\mathcal{C}|$ denotes the cardinality of $\mathcal{C}$. The dynamics of the MAS are described by the LTI system:
    \begin{align}
        \dot{{\vct{{x}}}} & = A\vct{x} + \sum\nolimits_{i \in \mathcal{C}}B_i \vct{u}_i \, , \label{eq_mas_dyn1}
    \end{align} 
where $A \in \mathbb{R}^{n \times n}$, $B_i \in \mathbb{R}^{n \times m_i}$, $\vct{x} \in \mathbb{R}^n$ contains the states of the dynamics of all agents and couplings while $\vct{u} _i \in \mathbb{R}^{m_i}$ is the input of agent $i \in \mathcal{C}$. Each agent $i \in \mathcal{C}$ seeks to minimize:
\begin{equation} \label{eq_functional_player}
   f_i(\vct{x}_0, \vct{u}_1, ..., \vct{u}_{|\mathcal{C}|}) = \int_0^{\infty}\vct{x}^\top Q_i \vct{x} + \vct{u}^\top_i R_i \vct{u}_i \ dt \, ,
\end{equation}
where $Q_i \succeq 0$ and  $R_i \succ 0$ are matrices with appropriate dimensions, and $\vct{x}_0$ is the initial state. Note that \eqref{eq_functional_player} indicates that $f_i(\vct{x}_0, \vct{u}_1, ..., \vct{u}_{|\mathcal{C}|})$ is a function of all agents' strategies. This is because they affect the state vector $\vct{x}$ according to \eqref{eq_mas_dyn1}.  Consider $A$ and $B_i \ \forall i \in \mathcal{C}$ are unknown. The objective is to determine the set of static matrices $\mathcal{K} = \{K_1, ..., K_{|\mathcal{C}|} \}$ where $K_i \in \mathbb{R}^{m_i \times n}$ such that the state feedback strategy $\vct{u}_i = K_i \vct{x} \ \forall i \in \mathcal{C}$ attains a NE.

\subsection{Problem 2: Stabilizing output feedback strategies}
\label{sec_problem2}
Having determined the set of matrices $\mathcal{K}$ from Problem 1, consider that the output $\vct{y}_i \in \mathbb{R}^{p_i} \ \forall i \in \mathcal{C}$ is given by:
    \begin{align}
\vct{y}_{i} & = C_{i} \vct{x} + D_{i}\vct{d}_{i} \, , \label{eq_mas_dyn2}
    \end{align} 
where $C_i$ and $D_i$ are known matrices, and $\vct{d}_i \in L^{m_i}_{2[0, \infty)}$ is a finite-energy disturbance to the output of agent $i \in \mathcal{C}$. The objective is to determine the set of static matrices $\mathcal{F} = \{F_1, ..., F_{|\mathcal{C}|} \}$ where $F_i \in \mathbb{R}^{m_i \times p_i}$ such that if the agents use the output feedback strategy $\vct{u}_i = F_i \vct{y}_{i} \ \ \forall  i\in \mathcal{C}$, then the closed-loop system is stable, i.e., $\lambda(A+ \sum_{i \in \mathcal{C}} B_iF_iC_i) \in \mathbb{C}_{-}$, and 
the closed-loop system is dissipative with supply function $\vct{d}^\top R\vct{d}$, where $\vct{d}=[\vct{d}_1^\top \ ... \ \vct{d}_{|\mathcal{C}|}^\top]^\top$ and $R = \mathrm{diag}(R_1, ...., R_{|\mathcal{C}|})$. Moreover, using $\mathcal{F}$, the difference between $f^{\mathcal{F}}_i$ and $f^{\mathrm{Nash}}_i \ \forall i \in \mathcal{C}$ is minimized, where $f^{\mathcal{F}}_i \in \mathbb{R}_{\geq 0}$ is the cost function of  agent $i \in \mathcal{C}$ that results from using $\vct{u}_i = F_i \vct{y}_i \ \forall i \in \mathcal{C}$ in \eqref{eq_functional_player}, and $f^{\mathrm{Nash}} _i\in \mathbb{R}_{\geq0}$ is the cost function of agent $i \in \mathcal{C}$ that results from using $\mathcal{K}$ (NE strategies) from Problem 1. 

\section{Control Framework}
\subsection{Preliminaries}
We first review a data-driven approach to solve the LQR problem \cite{depersis_lqr_datadriven, depersis_robust_lqr, dai_output_datadriven}, and then propose Corollary \ref{cor_recover_p} to determine the solution of an algebraic Riccati equation (ARE). In the LQR, there is a single agent $\mathcal{C}=\{ 1 \}$, that seeks to minimize:
\begin{equation} \label{eq_f_one_player}
   f(\vct{x}_0, \vct{u}) = \int_0^{\infty}\vct{x}^\top Q \vct{x} + \vct{u}^\top R \vct{u} \ dt \, ,
\end{equation}
by using a state feedback strategy $\vct{u} = K \vct{x}$ subject to:
\begin{align}
    \vct{\dot{x}} = A\vct{x} + B\vct{u} \, . \label{eq_oneplayer_system}
\end{align}
Solving the LQR we obtain $K = R^{-1} B^\top P$, $K \in \mathbb{R}^{m \times n}$, where $P$ is the solution to the ARE:
\begin{align}
    A^\top P + P A - P B R^{-1}B^\top P + Q = 0_{n \times n} \, . \label{eq_are}
\end{align}
However, if $A$ and $B$ are unknown, \eqref{eq_are} cannot be solved directly. To circumvent this, a data-driven approach is proposed in \cite{depersis_lqr_datadriven} to determine $K$. First, the system \eqref{eq_oneplayer_system} is excited via $\vct{u}$ with a persistently exciting signal \cite{willems_persist}. Data samples obtained from the input and state trajectories are used to construct:
\begin{align}
    U  & = \begin{bmatrix}
        \vct{u}(0) & \vct{u}(\Delta) & \vct{u}(2\Delta) & ... & \vct{u}(t_{end})
    \end{bmatrix} \in \mathbb{R}^{m \times l} \, ,  \label{eq_U_lqr}\\
    X  & = \begin{bmatrix}
        \vct{x}(0) & \vct{x}(\Delta) & \vct{x}(2\Delta) & ... & \vct{x}(t_{end})
    \end{bmatrix} \in \mathbb{R}^{n \times l } \, ,\label{eq_X_lqr}
\end{align}
where $l = \frac{t_{end}}{\Delta} +1 $, $\Delta \in \mathbb{R}_{>0}$ is the sampling time, and $t_{end} \in \mathbb{R}_{>0}$ is the end time of the experiment. Using the data in $X$, the following matrix can be constructed using finite differences or other numerical approximations:
\begin{align} 
    \widetilde{X}  = \begin{bmatrix}
        \vct{\dot{x}}(0) & \vct{\dot{x}}(\Delta) & \vct{\dot{x}}(2\Delta) & ... & \vct{\dot{x}}(l-1)
    \end{bmatrix} \in \mathbb{R}^{n \times l } \, .\label{eq_dX_lqr}
\end{align}
Finally, $K$ can be determined using $K = UZ W^{-1}$, where $Z$ and $W$ are a solution to the following SDP \cite{depersis_lqr_datadriven}:
 \begin{align}
       & \min_{\substack{W \in \mathbb{R}^{n \times n}, T \in \mathbb{R}^{m \times m}, Z \in \mathbb{R}^{l \times n}}} \mathrm{trace}(QW) +  \mathrm{trace}(T) \,, \label{eq_lqr_sdp_obj} 
       \end{align}
subject to:
\begin{gather}
    \widetilde{X} Z + Z^\top \widetilde{X}^\top + I_{n} \preceq 0,  \label{eq_lqr_sdp_ly} \\
    \begin{bmatrix}
        T & R^{{1}/{2}} U Z\\
        (R^{{1}/{2}} U Z)^\top & W 
        \end{bmatrix} \succeq 0, \label{eq_lqr_sdp_tr}  \\
        X Z = W. \label{eq_lqr_sdp_t} 
\end{gather}
Using the solution $W$, $Z$, and $T$, we propose the following corollary to determine the solution to the ARE \eqref{eq_are} when $A$ and $B$ are unknown.
 \begin{corollary}
\label{cor_recover_p}
Given $M = \widetilde{X} Z W^{-1}$ and $N = -Q -K^\top R K$, the unique solution $P$ to the ARE \eqref{eq_are} can be obtained by solving the following Sylvester equation for $P$:
\begin{align}
    \label{eq_syl_p}
    M^\top P +  P M = N \, .  
    \end{align}
\end{corollary}
\begin{proof}
By \cite[Th. 2]{depersis_lqr_datadriven}, the closed-loop state matrix satisfies: 
\begin{align}
     A + B K = \widetilde{X} Z W^{-1} \, , \label{eq_abk_lqr}
\end{align}
where $\widetilde{X}$ is given in \eqref{eq_dX_lqr}, $W$ and $Z$ satisfy the SDP \eqref{eq_lqr_sdp_obj}-\eqref{eq_lqr_sdp_t}. Furthermore, using $K = -R^{-1} B^\top P$ and \eqref{eq_abk_lqr} in \eqref{eq_are} results
\begin{align}
    (\widetilde{X} Z W^{-1})^\top P + P (\widetilde{X} Z W^{-1}) = -Q - K^\top R K \, , \label{eq_are_re}
\end{align}
 which is equation \eqref{eq_syl_p}. By \cite[Lem. 2.7]{zhou_control}, the Sylvester equation \eqref{eq_are_re} has unique solution $P$ if and only if $\lambda(\widetilde{X} Z W^{-1}) \cap \lambda(-\widetilde{X} Z W^{-1}) = \emptyset$. This is true because $\lambda(A+BK) \in \mathbb{C}_{-}$, which by \eqref{eq_abk_lqr}, implies $\lambda(\widetilde{X} Z W^{-1}) \in \mathbb{C}_{-}$.
\end{proof}

\subsection{Nash equilibrium}
\label{sec_ne}
The NE concept can be used to solve a non-cooperative linear quadratic differential game \cite{lineargames_engwerda, thulasi_constructive_nash}. In these games, there is a set of agents $\mathcal{C}$, and each agent $i$ uses a strategy $\vct{u}_i = K_i\vct{x}$ to minimize its individual cost function \eqref{eq_functional_player} subject to \eqref{eq_mas_dyn1}. To find a NE, it requires solving for a set of matrices $\mathcal{P} = \{P_1, ..., P_{|\mathcal{C}|} \}$ that satisfies a series of algebraic Riccati equations (CAREs) such that $\mathcal{R}_{i}(\mathcal{P}) = 0_{n \times n} \ \forall i \in \mathcal{C}$, where:
\begin{align}
& \mathcal{R}_{i}(\mathcal{P}) = \Big(
{A}- \sum_{j \in \mathcal{C}_{-i} } B_jR_j^{-1}B_j^\top P_{j} \Big)^\top P_{i} \notag \\ 
&+ P_{i} \Big({A} - \sum_{j \in \mathcal{C}_{-i} }   B_jR_j^{-1}B_j^\top P_{j} \Big) - P_{i} B_iR_i^{-1}B_i^\top P_{i} + Q_{i} \, , \label{eq_coupled_ricky}
\end{align}
where $\mathcal{C}_{-i}$ denotes the set of agents $\mathcal{C}$ excluding the agent $i$. Next, the set of matrices $\mathcal{K}$ is determined using 
$K_i = R_i^{-1}B^\top_i P_i \ \forall i \in \mathcal{C}$. Using $\vct{u}_i = K_i \vct{x} \ \forall i \in \mathcal{C}$, the MAS \eqref{eq_mas_dyn1} attains a NE. However, as $A$ and $B_i \ \forall i \in \mathcal{C}$ are unknown, model-based methods in \cite{capello_nash_approx, Engwerda2007} cannot be used to solve the CAREs. In turn, we propose Algorithm \ref{algo1} to determine the sets $\mathcal{K}$ and $\mathcal{P}$ using data. 
\SetKwInput{KwInput}{Input}
\SetKwFor{ParFor}{paralell for}{do}{endfor} 
\begin{algorithm}
\KwInput{$Q_i \succeq 0, R_i \succ 0 \ \forall i \in \mathcal{C}$, $\Delta \in \mathbb{R}_{>0}$, $\epsilon_{{CARE}} \in \mathbb{R}_{>0}$, ${r}_{max} \in \mathbb{N}$}
\KwResult{$\mathcal{K} = \{K_1, ..., K_{|\mathcal{C}|} \}$ and $\mathcal{P} = \{P_1, ..., P_{|\mathcal{C}|} \}$.}
\For{$r=1$ to $r_{{max}}$}{
    \ParFor{$i=1$ to $|\mathcal{C}|$}{
    \If{$r=1$}{
    Excite the MAS \eqref{eq_mas_dyn1} with $\vct{u}_{j} = 0 \ \forall j \in \mathcal{C}_{-i}$ and a persistently exciting $\vct{u}_{i}$. \\
    }
    \Else{Compute $\vct{u}_{j}$ using $\vct{x}$ in $\vct{u}_{j} = K^{(r-1)}_j \vct{x} \ \forall j \in \mathcal{C}_{-i}$ in the MAS \eqref{eq_mas_dyn1}. \\
    Excite the closed-loop MAS with a persistently exciting $\vct{u}_{i}$.}
    Construct $U^{(r)}_i$ and $X^{(r)}$ using \eqref{eq_U_lqr}-\eqref{eq_X_lqr}. \\
    Compute $\widetilde{X}^{(r)}$ using \eqref{eq_dX_lqr}. \\
    Compute $W^{(r)}_i$, $T^{(r)}_i$ and $Z^{(r)}_i$ by solving \eqref{eq_lqr_sdp_obj}-\eqref{eq_lqr_sdp_t}. \\ 
    Compute $K^{(r)}_i = U^{(r)}_i Z^{(r)}_i (W^{(r)}_i){}^{-1}$. \\
    Compute $P^{(r)}_i$ using $\widetilde{X}^{(r)}, Z^{(r)}_i, W^{(r)}_i, Q_i, R_i$, and $K^{(r)}_i$ in \eqref{eq_syl_p}.\\
    Compute $\epsilon^{(r)}_i = \| K^{(r)}_{i} - K^{(r-1)}_{i} \|$. \\
    }
    \If{$r>1$ {and} $\max({\epsilon_1, ..., \epsilon_{|\mathcal{C}|}}) \leq \epsilon_{CARE}$}{
        Set $\mathcal{K} = \{K^{(r)}_1, ..., K^{(r)}_{|\mathcal{C}|} \}$. \\
        Set $\mathcal{P} = \{P^{(r)}_1, ..., P^{(r)}_{|\mathcal{C}|} \}$. \\
        Stop Algorithm.}        
    }
\caption{Nash seeking iterative LQR method}
\label{algo1}
\end{algorithm}
The input parameters to Algorithm \ref{algo1} are $Q_i$ and $R_i$ for each agent $i$, a sampling time $\Delta$, a tolerance $\epsilon_{{CARE}}$, and a maximum number of iterations ${r}_{max}$. 

The first iteration ($r=1$) for agent $i$ consists of exciting the MAS \eqref{eq_mas_dyn1} with a persistently exciting signal via $\vct{u}_i$. The inputs $\vct{u}_j \ \forall j \in \mathcal{C}_{-i}$ are fixed to zero. From this experiment, we construct the matrices $U^{(r)}_i$, $X^{(r)}_i$ and $\widetilde{X}^{(r)}_i$, which are used as inputs to the SDP \eqref{eq_lqr_sdp_obj}-\eqref{eq_lqr_sdp_t}, to determine $W_i$, $T_i$, $Z_i$, and then $K_i$. Next, Corollary \ref{cor_recover_p} is used to determine $P_i$. Once $K^{(r)}_i$ is determined for all the agents, $K^{(r)}$ is formed by vertically concatenating all $K^{(r)}_i$. 

In the next iterations (i.e., $r>1$) for agent $i \in \mathcal{C}$, the MAS must be connected in closed-loop in which all inputs except $\vct{u}_i$ use the state feedback strategy computed in the previous iteration.  Then, the MAS is excited via $\vct{u}_i$ with a persistently exciting signal. The subsequent steps are same as the iteration $r=1$. All this procedure can be run in a parallel for-loop for all the agents. If $\epsilon_{i} \leq \epsilon_{CARE} \ \ \forall i \in \mathcal{C}$ and assuming that $\epsilon_{CARE}$ is a small number such as $10^{-3}$, then the solution to the CAREs is given by the sets $\mathcal{P}$ and $\mathcal{K}$ of the last iteration.
The next corollary proves that the resulting set $\mathcal{P}$ of Algorithm \ref{algo1} satisfies the set of CAREs. 
\begin{corollary}
    If $\lim_{r \to \infty} \epsilon_i^{(r)} = 0 \ \forall i \in \mathcal{C}$, $\mathcal{P}^{(\infty)}$ solves $\mathcal{R}_{i}(\mathcal{P}) = 0_{n \times n} \ \forall i \in \mathcal{C}$, where $\mathcal{R}_{i}$ is given by \eqref{eq_coupled_ricky}.
\end{corollary}
\begin{proof}
    At iteration $r$, $K^{(r-1)}_{i} \ \forall i \in  \mathcal{C}$ has been determined, and the closed-loop state matrix is $A^{(r-1)}_{i} = {A}+ \sum_{j \in \mathcal{C}_{-i} } B_jK^{(r-1)}_j$. Then, the SDP \eqref{eq_lqr_sdp_obj}-\eqref{eq_lqr_sdp_t} yields $K^{(r)}_{i} = R_i^{-1}B_i^\top P^{(r)}_{i}$. By Corollary \ref{cor_recover_p}, $K^{(r)}_{i}$ satisfies \eqref{eq_are} where $A = A^{(r-1)}_i$, $B = B_i$, $R = R_i$, and $Q = Q_i$.    
    As $\lim_{r \to \infty} \epsilon_i^{(r)} = 0$, $P^{(\infty)}_{i} = P_i$ and $K^{(\infty)}_{i} = K_i \ \forall i \in \mathcal{C}$. Using these two equalities and taking $r \to \infty$, \eqref{eq_are} with $A = A^{(\infty)}_i$, $B = B_i$, $R = R_i$, and $Q = Q_i \ \forall i \in \mathcal{C}$ becomes the set of CAREs.
\end{proof}

It worth mentioning that \cite[Algo. 2]{nortmann_nash_iterative} meticulously develops a similar scheme as in Algorithm \ref{algo1}. However, our Algorithm \ref{algo1} incorporates two novel features: equation \eqref{eq_syl_p} of Corollary \ref{cor_recover_p} and a parallel for-loop. While the former feature enables the determination of output feedback strategies as Section \ref{sec_dd_output_feedback} will show, the later feature improves the convergence speed of the Algorithm \ref{algo1} as Section \ref{sec_case_studies} illustrates. 

We also notice that although Algorithm \ref{algo1} has the benefit of not relying on explicitly identifying the matrices $A$ and $B_i \ \forall i \in \mathcal{C}$, it has the drawback that it assumes that $\begin{bmatrix}
    A & B_i
\end{bmatrix}$ is stabilizable for every $i \in \mathcal{C}$. This assumption is indeed stronger than having $\begin{bmatrix}
    A & B_1 & ... & B_{|\mathcal{C}|}
\end{bmatrix}$ stabilizable, which is needed to ensure the existence of a NE  \cite{capello_nash_approx, Engwerda2007}. As the assumption of having $\begin{bmatrix}
    A & B_i
\end{bmatrix}$ stabilizable for every $i \in \mathcal{C}$ can be difficult to meet in some MAS, we develop Algorithm \ref{algo2}, which does not require this assumption. 

\begin{algorithm}
\KwInput{$Q_i \succeq 0, R_i \succ 0 \ \forall i \in \mathcal{C}$, $\Delta \in \mathbb{R}_{>0}$ $\epsilon_{\mathrm{CARE}} \in \mathbb{R}_{>0}$, ${r}_{max} \in \mathbb{N}$}
\KwResult{$\mathcal{K}^{} = \{K^{}_1, ..., K^{}_{|\mathcal{C}|} \}$ and $\mathcal{P}^{} = \{P^{}_1, ..., P^{}_{|\mathcal{C}|} \}$.}
Define $Q = \frac{1}{|\mathcal{C}|}\sum_{i \in \mathcal{C}} Q_i$ and $R = \mathrm{diag}(R_1, ..., R_2)$.\\
Excite the MAS \eqref{eq_mas_dyn1} with persistently exciting signals via $\vct{u}^{}_1 , \ ... \ , \vct{u}^{}_{|\mathcal{C}|}$. \\
Construct $U^{(r=1)}_i$ and $X^{(r=1)}$ using \eqref{eq_U_lqr}-\eqref{eq_X_lqr}. \\
Compute $\widetilde{X}^{(r=1)}$ using \eqref{eq_dX_lqr}. \\
Compute $B^{est}_{i}$ using \eqref{eq_b_est}. \\
Compute $S_i = B^{est}_{i} R^{-1}_{i} B^{est\top}_{i} \ \forall i \in \mathcal{C}$. \\
Compute $W^{(r=1)}$, $T^{(r=1)}$ and $Z^{(r=1)}$ by solving \eqref{eq_lqr_sdp_obj}-\eqref{eq_lqr_sdp_t}. \\ 
Compute $K^{(r=1)} = U^{(r=1)} Z^{(r=1)} (W^{(r=1)}){}^{-1}$. \\
Compute $P^{(r=1)}$ using $\widetilde{X}^{(r=1)}, Z^{(r=1)}, W^{(r=1)}, Q, R$, and $K^{(r=1)}$ in \eqref{eq_syl_p}.\\
Compute $A_c^{(r=1)} = \widetilde{X}^{(r=1)} Z^{(r=1)} (W^{(r=1)}){}^{-1}$. \\
Set $P_i^{(r=1)} =  P^{(r=1)} \ \forall i \in \mathcal{C}$.\\
\For{$r=2$ to $r_{max}$}{
    Compute $\alpha^{(r)}, {P}^{(r)}_1, ..., P^{(r)}_{|\mathcal{C}|}$ by solving \eqref{eq_nw_sdp_obj}-\eqref{eq_nw_p}. \\
    Compute ${K}^{(r)}_i = R^{-1}_i B^{est}_i {}^\top P^{(r)}_i \ \forall i \in \mathcal{C}$. \\
    Compute $A^{(r)}_c$ by solving \eqref{eq_invlqr_obj}-\eqref{eq_invlqr_w}. \\
    Compute $\epsilon^{(r)}_i = \| P^{(r)}_{i} - P^{(r-1)}_{i} \| \ \forall i \in \mathcal{C}$ . \\
    \If{$\max({\epsilon_1, ..., \epsilon_{|\mathcal{C}|}}) \leq \epsilon_{CARE}$}{
        Set $\mathcal{K} = \{K^{(r)}_1, ..., K^{(r)}_{|\mathcal{C}|} \}$. \\
        Set $\mathcal{P} = \{P^{(r)}_1, ..., P^{(r)}_{|\mathcal{C}|} \}$. \\
        Stop Algorithm.}}
\caption{Nash seeking Newton-Raphson method}
\label{algo2}
\end{algorithm}
 
Inspired by \cite{Engwerda2007}, in Algorithm \ref{algo2}, we solve $\mathcal{R}_{i}(\mathcal{P}^{(r)}) = 0_{n \times n} \ \forall i \in \mathcal{C}$ at each iteration $r$. By replacing $\mathcal{P}^{(r)} = \{ P^{(r-1)}_{1} + \delta^{(r-1)}_1, ..., P^{(r-1)}_{|\mathcal{C}|} + \delta^{(r-1)}_{|\mathcal{C}|} \}$  in $\mathcal{R}_{i}(\mathcal{P}^{(r)}) = 0_{n \times n} \ \forall i \in \mathcal{C}$ and using $\delta^{(r)}_i = P^{(r)}_i - P^{(r-1)}_i$ results in:
\begin{align}
    A^{(r-1)}_c {}^{\top} P^{(r)}_{i} + P^{(r)}_{i}A^{(r-1)}_c  + \sum\limits_{j \in \mathcal{C}_{-i}} N^{(r)}_j + N^{(r)}_j{}^\top \notag \\ 
      + Q_i + P_i^{(r)} S_i P_i^{(r)}= 0_{n \times n} \, ,\label{eq_nw_approach}
\end{align}
$\forall i \in \mathcal{C}$, where $A^{(r-1)}_c = A + \sum_{j \in \mathcal{C}}S_jP^{(r-1)}_j$, $S_j = B_j R^{-1}_jB_{j}^\top \ \forall j \in \mathcal{C}$, and  $N^{(r)}_j =P_j^{(r-1)}S_{j} P^{(r-1)}_i - P^{(r)}_j S_j P^{(r-1)}_i $. Note that $P^{(r)}_i \ \forall i \in \mathcal{C}$ is an unknown variable, whereas $P^{(r-1)}_i \ \forall i \in \mathcal{C}$ is known as it results from the previous iteration $r-1$. We define the left-hand side of \eqref{eq_nw_approach} as ${T}^{(r)}_{i}$ and propose the next SDP:
\begin{align}
    & \min_{\substack{\alpha \in \mathbb{R}_{\geq 0}, {P}^{(r)}_1, ..., P^{(r)}_{|\mathcal{C}|} \in \mathbb{R}^{n \times n}}}  \alpha \, , \label{eq_nw_sdp_obj} 
    \end{align}
    subject to: 
    \begin{gather}
        | e_k^\top T^{(r)}_i e_q | \leq \alpha, \ \forall \ k \in \mathbb{N}_{\le q}\, , \ q  \in \mathbb{N}_{\le n},  \ i \in \mathcal{C} \, , \label{eq_abs_T} \\
        P^{(r)}_{i} \succeq 0 \, , \ \forall i \in \mathcal{C} \label{eq_nw_p}
    \end{gather}
where $e_k = \begin{bmatrix}
    0 & \dots & 1 & \dots & 0
\end{bmatrix}^\top \in \mathbb{R}^n$ and its $k$-th element is $1$. This SDP aims to make every entry of $T^{(r)}_i$ equal to zero so that \eqref{eq_nw_approach} holds. As it is shown in Algorithm \ref{algo2}, this SDP becomes a fundamental part to find a NE.

Algorithm \ref{algo2} begins with exciting the MAS \eqref{eq_mas_dyn1} with persistently exciting signals via all input channels. By \cite[Th. 1]{willems_note_excitation}, the input matrix for each agent $B^{est}_i$ is estimated using:
\begin{align}
\begin{bmatrix}
    B^{est}_1 & ... &  B^{est}_{|\mathcal{C}|}
\end{bmatrix}  = \begin{bmatrix}
    \widetilde{X}^{(r)} \begin{bmatrix}
        U^{(r)} \\ X^{(r)}
    \end{bmatrix}
\end{bmatrix}_{(:, 1:m)} \, . \label{eq_b_est}
\end{align}
Lines 9-11 solve the LQR and use Corollary \ref{cor_recover_p}. Then, the closed-loop state matrix $A^{(r)}_c$ is computed using \eqref{eq_abk_lqr}. Next, $P^{(r)}_i \ \forall i \in \mathcal{C}$ is set to be equal to the solution $P^{(r)}$ determined in line 11. In the next iterations, the SDP \eqref{eq_nw_sdp_obj}-\eqref{eq_nw_p} uses $P^{(r-1)}_i \ \forall i \in \mathcal{C}$ and $A^{(r-1)}_c$ as input data to compute $\alpha^{(r)}$, $P^{(r)}_i \ \forall i \in \mathcal{C}$. Next, we compute $K^{(r)}_i \ \forall i \in \mathcal{C}$ using $B^{est}_i$, which has been computed once in the first iteration for all the agents. Since the SDP \eqref{eq_lqr_sdp_obj}-\eqref{eq_lqr_sdp_t} is not used for iteration $r>1$, then \eqref{eq_abk_lqr} cannot be used directly to calculate $A_c^{(r)}$. In turn, we propose the SDP in Corollary \ref{coro_3} to determine  $A_c^{(r)}$.
\begin{corollary} \label{coro_3}
Given ${U}^{(r=1)}$, ${X}^{(r=1)}$, $\widetilde{X}^{(r=1)}$, define:
\begin{align}
    K^{(r)} = \begin{bmatrix}
        K_1^{(r) \ \top} & ... & K_{|\mathcal{C}|}^{(r) \ \top}
    \end{bmatrix}^\top \, . \label{k_total}
\end{align}
The closed-loop state matrix $A^{(r)}_c$ can be computed as $A^{(r)}_c = \widetilde{X}^{(r=1)} Z^{(r)} (W^{(r)}) {}^{-1}$, where $Z^{(r)}$ and $W^{(r)}$ are the solution to:
    \begin{align}
       & \min_{\substack{W^{(r)} \in \mathbb{R}^{n \times n}, Z^{(r)} \in \mathbb{R}^{l \times n}}} \mathrm{trace}(W^{(r)}) \label{eq_invlqr_obj} 
\end{align}
subject to:
\begin{gather}
     \widetilde{X}^{(r=1)} Z^{(r)} + Z^\top \widetilde{X}^{(r=1)} {}^\top + I_{n} \preceq 0 \, , \label{eq_invlqr_wr} \\
 \begin{bmatrix}
        K^{(r)} \\ I_n
    \end{bmatrix} W^{(r)} = \begin{bmatrix}
        U^{(r=1)} \\ X^{(r=1)}
    \end{bmatrix} Z^{(r)} \, , \label{eq_invlqr_ki} \\
     W^{(r)} \succeq 0 \, . \label{eq_invlqr_w} 
\end{gather}
\end{corollary}
\begin{proof}
As $K^{(r)}_i$ satisfies \eqref{eq_nw_approach}, $\lambda (A_c^{(r)}) \in \mathbb{C}_-$, where $A^{(r)}_c = A + B K^{(r)}$ and $B = [B_1 \ ... \ B_{|\mathcal{C}|}]$. Moreover, define the system $\hat{F}$ as:
\begin{align}
    \vct{\dot{x}} & = (A + BK^{(r)}) \vct{x} + I_n \vct{u}, \\ 
    \vct{y} & = C \vct{x} \, .
\end{align}
By bounded input-bounded output stability, since $\lambda (A_c^{(r)}) \in \mathbb{C}_-$, then it exists $\epsilon>0$ such that $\|\hat{F}\|_{\mathcal{H}_2} < \epsilon $, where $\mathcal{H}_2$ is the two-to-infinity norm. By \cite[Prop. 6.13]{Dullerud_convex_book}, there exists $W \succ 0$ where $\mathrm{trace}(C W^{(r)} C^\top) < \epsilon$ and
\begin{align}
    (A + B K^{(r)}) W^{(r)} + W^{(r)} (A+BK^{(r)})^\top + I_n \preceq 0 \, . \label{eq_abk_w}
\end{align}
By \cite[Th. 2]{depersis_lqr_datadriven}, $A + B K^{(r)} = \widetilde{X}^{(r)}G_{K}^{(r)}$, where $G_{K}^{(r)}$ satisfies:
\begin{align}
     \begin{bmatrix}
        K^{(r)} \\ I_n
    \end{bmatrix} = \begin{bmatrix}
        U^{(r=1)} \\ X^{(r=1)}
    \end{bmatrix} G_K^{(r)} \, . \label{eq_th_g}
\end{align}
Defining $G^{(r)}_{K} = Z^{(r)} (W^{(r)})^{-1}$, \eqref{eq_abk_w}-\eqref{eq_th_g} become \eqref{eq_invlqr_wr}-\eqref{eq_invlqr_w}. 
\end{proof}

\subsection{Stabilizing output-feedback strategies}
\label{sec_dd_output_feedback}
Using the sets $\mathcal{K}$ and $\mathcal{P}$ that result from Algorithm \ref{algo1}, we propose the following Lemma \ref{lemma_sdp_problem_data} that determines the set $\mathcal{F} = \{F_1, ..., F_{|\mathcal{C}|} \}$ for the output feedback strategies $\vct{u}_i = F_i \vct{y}_{i} \ \forall i \in \mathcal{C}$ for the MAS \eqref{eq_mas_dyn1}-\eqref{eq_mas_dyn2}. Note that Lemma \ref{lemma_sdp_problem_data} does not require knowledge of $A$ and $B_i \ \forall i \in \mathcal{C}$.  
\begin{lemma}
\label{lemma_sdp_problem_data}
Given the sets of matrices $\mathcal{K}=\{K_1, ..., K_{|\mathcal{C}|}\}$ and $\mathcal{P}=\{P_1, ..., P_{|\mathcal{C}|}\}$ that result from either Algorithm \ref{algo1} or Algorithm \ref{algo2}, and the parameters $a$, $b$, and $c \in \mathbb{R}_{>0}$. If the optimal solution of the following SDP:
 \begin{align}
       & \ \ \ \ \ \ \ \ \ \min_{\substack{\alpha_i, \beta_i, \gamma_i, F_i \\ \forall  \ i \in \mathcal{C}}}  \sum _{i \in \mathcal{C}} -a \alpha_i + b \beta_i + c \gamma_i \, , \label{eq_sdp_obj_d} \\
      & \text{subject to: }\notag \\ 
     & \begin{bmatrix}
        \Gamma_i + \Upsilon_i & \bar{K}^{\top}_i & \Omega_{-i}  \\
        \bar{K}_i  & R_i & 0 \\
        \Omega_{-i}^\top & 0 & \bar{R}_{-i}
    \end{bmatrix}  \succeq \alpha_i I_{n+m} \  \ \forall  \ i \in \mathcal{C} \, , \label{eq_lemma1_c_d} \\
    & \begin{bmatrix}
        \beta_i I_{m_i} & -R_i^{\frac{1}{2}}K_i + R_i^{\frac{1}{2}} F_{i} C_{i} \\
        ( -R_i^{\frac{1}{2}}K_i + R_i^{\frac{1}{2}} F_{i} C_{i})^\top & \beta_i I_{n}
    \end{bmatrix} \succeq 0 \notag \\
    &  \  \ \ \ \ \ \ \ \ \ \ \ \ \ \ \ \ \ \ \ \ \ \ \ \ \ \ \ \ \ \ \ \ \ \ \ \ \ \ \ \ \ \ \ \ \ \ \ \ \ \ \ \forall  \ i \in \mathcal{C}, \label{eq_nashapprox_d}\\
     & \begin{bmatrix}
        \gamma_i I_{m_i} & R_i^{\frac{1}{2}} F_i D_i \\
        (R_i^{\frac{1}{2}} F_i D_i)^\top & \gamma_i I_{m_i}
    \end{bmatrix} \succeq 0 \  \ \forall  \ i \in \mathcal{C} \, , \label{eq_atten_d}
 \end{align}
 returns $\alpha_i>0$ for at least one $i \in \mathcal{C}$, then the resulting set of matrices $\mathcal{F} = \{F_1, ..., F_{|\mathcal{C}|} \}$ forms the output feedback control $\vct{u}_i = F_i \vct{y}_i \  \forall i \in \mathcal{C}$ that satisfies the objectives of Problem $2$ in Section \ref{sec_problem2}. In \eqref{eq_lemma1_c_d}-\eqref{eq_atten_d}, the matrices $\Gamma_i \in \mathbb{R}^{n \times n}$, $\Upsilon_i \in \mathbb{R}^{n \times n}$, $\bar{K}_i \in \mathbb{R}^{m_i \times n}$, $\Omega_{-i} \in \mathbb{R}^{n\times m_i(|\mathcal{C}|-1)}$, and $\bar{R}_{-i} \in \mathbb{R}^{m_i(|\mathcal{C}|-1) \times m_i(|\mathcal{C}|-1)}$ are defined as:
\begin{align}
    & \Gamma_i = \notag \\
    & \text{\small $Q_i + \sum_{j \in \mathcal{C}_{-i}} \Big( L_{ij}K_j + K_j^{\top} L^\top_{ij} \Big) -\sum_{j \in \mathcal{C}}  \Big( L_{ij} F_j C_j + C_j^\top F^{\top}_j L^{\top}_{ij}\Big)$}, \label{def_gamma_datadriven}\\
    & \Upsilon_i  = K^\top_i R_i \left(F_i D_i R_i^{-1} + R_i^{-1}D_i^\top F_i^\top \right) R_i K_i \, , \label{def_upsilon_datadriven} \\
    & \bar{K}_i = (I_{m_i} + D_i^\top F_i^\top)R_i K_i \, , \label{eq_def_kibar} \\
    & \Omega_{-i} = [Z_{i1} \ ... \ Z_{i(i-1)} \quad Z_{i(i+1)}  \ ... \ Z_{i|\mathcal{C}|}], \\
    & Z_{ij} = L_{ij} F_j D_j,\\
    & \bar{R}_{-i} =  \mathrm{diag}(R_1, ..., R_{i-1}, R_{i+1}, ..., R_{|\mathcal{C}|}) \, ,
\end{align}
where:
\begin{align}
    L_{ij} = P_i P_j^{-1} K^{\top}_j R_j \ \forall \ i,j \in \mathcal{C} \, . \label{eq_lij_def}
\end{align}
\end{lemma}
\begin{proof}
    In Lemma 3 from \cite{nash_outputfeedback}, we have proposed a model-based SDP to synthesize the set $\mathcal{F}$. In this proof, we reformulate the SDP such that it only requires the sets $\mathcal{K}$ and $\mathcal{P}$ from Algorithm \ref{algo1}. First, using $K_j = - R^{-1}_jB^\top_j P_j \ \forall K_j \in \mathcal{K}, j \in \mathcal{C}$ \cite{capello_nash_approx, Engwerda_LQbook}, we obtain:
    \begin{align}
        P_j B_j  = - K^\top_j R_j \, . \label{eq_pb}
    \end{align}
    Using \eqref{eq_pb} and defining $L_{ij} =  P_i B_j$ results in \eqref{eq_lij_def}.
    In Lemma 3 from \cite{nash_outputfeedback}, we define:
    \begin{align}
        \Gamma_i = Q_{i} - P_{i} \Big( \sum_{j \in \mathcal{C}_{-i} }  B_jR_j^{-1}B_j^\top P_j + \sum_{ j \in \mathcal{C}}  B_jF_jC_j \Big)^\top \notag \\
    - \Big( \sum_{j \in \mathcal{C}_{-i} }  P_jB_jR_j^{-1} B_j^\top + \sum_{ j \in \mathcal{C}}  C_j^\top F_j^\top B_j^\top \Big) P_i \, . \label{eq_sdp1_prev}
    \end{align}
    Using \eqref{eq_lij_def} in \eqref{eq_sdp1_prev} results in \eqref{def_gamma_datadriven}. Similarly, in Lemma 3 from \cite{nash_outputfeedback}, it is defined that:
    \begin{align}
        \Upsilon_i = P_i B_i \left(F_i D_i R_i^{-1} + R_i^{-1}D_i^\top F_i^\top \right) B_i^\top P_i \, .\label{eq_sdp2_prev}
    \end{align}
    Using \eqref{eq_pb} in \eqref{eq_sdp2_prev} results in \eqref{def_upsilon_datadriven}. In Lemma 3 in \cite{nash_outputfeedback}, the first-row second-column matrix entry in \eqref{eq_lemma1_c_d} is defined as $P_i B_i(I_{m_i} + F_i D_i)$. Using \eqref{eq_def_kibar} and \eqref{eq_pb} results in $-\bar{K}_i^{\top} = P_i B_i(I_{m_i} + F_i D_i)$. The negative sign is removed in \eqref{eq_lemma1_c_d} because:
    \begin{align}
        \begin{bmatrix}
            A & B \\
            B^\top & C
        \end{bmatrix} \succeq 0 \text{ if and only if } \begin{bmatrix}
            A & -B \\
            -B^\top & C
        \end{bmatrix} \succeq 0 \, .
    \end{align}  
    In Lemma 3 from \cite{nash_outputfeedback}, the first-row third-column matrix entry in \eqref{eq_lemma1_c_d} is defined as $P_i\bar{D}_{-i}$, where: 
    \begin{align}
    & \text{\small $\bar{D}_{-i} =$} \notag \\
    & \text{\small $[B_1F_1D_1 \ ... \ B_{i-1}F_{i-1}D_{i-1} \ \ B_{i+1}F_{i+1}D_{i+1} \ ... \ B_{|\mathcal{C}|}F_{|\mathcal{C}|}D_{|\mathcal{C}|}]$} \, .
    \end{align}
    Using \eqref{eq_lij_def} results in $P_i\bar{D}_{-i} = \Omega_{-i}$. In Lemma 3 from \cite{nash_outputfeedback},  first-row second-column matrix entry in the constraint \eqref{eq_nashapprox_d} is defined as $R_i^{-\frac{1}{2}}B_i^\top P_i + R_i^{\frac{1}{2}} F_{i} C_{i}$, which by \eqref{eq_pb}, results in \eqref{eq_nashapprox_d}. Finally,  constraint \eqref{eq_atten_d} is identical to the corresponding constraint in Lemma 3 from \cite{nash_outputfeedback}. 
\end{proof}

\section{Case studies}
\label{sec_case_studies}
\subsection{Computational implementation}
We implement the \textcolor{black}{proposed framework in CMAS \cite{cmas_repo}. CMAS is an open-source toolbox in Python and Matlab for the control synthesis for MAS. We release the first version of CMAS in our prior work \cite{nash_outputfeedback}, implementing model-based control. Through this work, we release a second version of CMAS, implementing the data-driven framework proposed in this work.}     

\subsection{Numerical cases}
Consider the unstable fourth-order LTI: $\vct{\dot{x}} = A \vct{x} + \sum_{i \in {1,2}} B_i \vct{u}_i$, where $A$ and $B =[B_1 \ B_2]$ are given in the illustrative example in \cite{depersis_lqr_datadriven}. First, $\mathcal{K}$ and $\mathcal{F}$ are computed using \cite{nash_outputfeedback}, which provides model-based methods that explicitly require $A$ and $B$ as input data. We contrast these static controllers with those obtained using  Algorithm \ref{algo1}, \ref{algo2} and Lemma \ref{lemma_sdp_problem_data}. In the application of these algorithms, we generate random square signals that are applied to the input channels and use samples of the resulting state trajectories. In Table \ref{table_model_vs_dd}, we define the error as:
\begin{equation}
    \text{Error} = 100 \times \max \Big( |{(K^{\text{data-based}}_i }\oslash{K^{\text{model-based}}_i}) -1 |\Big),
\end{equation}
where $\max(A)$ is the maximum element of the matrix $A$, and $\oslash $ denotes element-wise division. Table \ref{table_model_vs_dd} shows that the error is less than 6\% across $\mathcal{K}$ and $\mathcal{P}$, demonstrating that our framework computes almost identical controllers to those obtained by model-based methods.

Furthermore, we generate ten random LTI systems and for each of these systems, we compare the convergence speed of Algorithm \ref{algo1} that uses parallel for-loop against the implementation in \cite[Algo. 2]{nortmann_nash_iterative} that uses traditional for-loop and synchronous updates. We see that our proposed implementation of Algorithm \ref{algo1} with parallel for-loop is on average $19 \%$ faster. 

\begin{table}[t]
\caption{Comparison of resulting controllers using the proposed framework and model-based methods \cite{nash_outputfeedback}.}
\begin{tabular}{ccccccc}
\toprule
  Data-based methods $\rightarrow$    & \multicolumn{2}{c}{Algorithm \ref{algo1}} & \multicolumn{2}{c}{Algorithm \ref{algo2}} & \multicolumn{2}{c}{Lemma \ref{lemma_sdp_problem_data}} \\
      & $K_1$           & $K_2$           & $K_1$           & $K_2$           & $F_1$        & $F_2$        \\ \cmidrule{2-7}
\begin{tabular}{c}
     Error ($\%$) relative to \\
     model-based methods.
\end{tabular}  & 0.62           & 0.03              & 1.27           & 0.37           & 5.99        & 0.06      \\
\bottomrule
\end{tabular} 
\label{table_model_vs_dd}
\end{table}


\begin{table}[t]
\centering
\caption{Comparison of performance between the proposed control framework and other methods}
\begin{tabular}{cccc}
\toprule
Case                                                                           & Control      & Overshoot$^{1}$ & Settling time$^{2}$ \\ 
\toprule
\multirow{4}{*}{\begin{tabular}[c]{@{}c@{}}Frequency \\ regulation \end{tabular}} & Proposed     & $0.0075$ Hz    & $0.568$ s        \\ \cmidrule{2-4} 
        & $k_i=40$   & $0.0344$ Hz   & $0.6392$ s     \\ \cmidrule{2-4} 
        & $k_i=80$     & $0.0318$ Hz   & $0.6802$ s       \\ \cmidrule{2-4} 
        & $k_i=120$   & $0.0302$ Hz   & $0.6377$ s       \\ \midrule
\multirow{3}{*}{\begin{tabular}[c]{@{}c@{}}Voltage\\ regulation \end{tabular}}  & Proposed     & $0.0098$ p.u.    & $0.0311$ s        \\ \cmidrule{2-4} 
        & Proportional-integral         & $0.0285$ p.u.    & $0.0427$ s        \\ \cmidrule{2-4} 
       & No control & $0.0383$ p.u.    & $0.0445$ s     \\  
       \bottomrule \\
       \multicolumn{4}{l}{\scriptsize{1, 2: Average of the overshoots or settling time for all disturbances and GFMI buses.}} \\
       \multicolumn{4}{l}
       {\scriptsize{1: Maximum value attained when disturbance is applied. }} \\
       \multicolumn{4}{l}
       {\scriptsize{2: Seconds until the absolute value of the derivative stayed below $0.01$ Hz/s}} \\
       \multicolumn{4}{l}{\scriptsize{and $0.1$ p.u./s.}}   
\end{tabular}
\label{table_performance}
\end{table}

\subsection{Frequency regulation}
\label{sec_freq_reg}
This case study develops a control scheme to regulate the frequency in a modified version of the IEEE 230-kV 9-bus system. We design an 94th-order electromagnetic transient (EMT) model in Simulink that consists of three grid-forming inverters (GFMIs), transmission lines and loads, shown in Fig. \ref{fig_nine_bus_system}. Each GFMI is a 13th-order model that includes a single-loop voltage controller, a virtual-inertia-based active power controller, a droop control for reactive power, a constant DC power supply, a pulse-width modulator (PWM), an average voltage source convert, and a LCL filter. Each transmission line has a $\pi$ structure which has a series RL branch and two parallel RC shunts. The loads are resistive shunts. Refer to \cite{sting_paper} for the initialization of the EMT scenario and a more detailed description of the dynamics of the components in this system. 

\begin{figure}[htb]
      \centering
      \includegraphics[width=\linewidth]{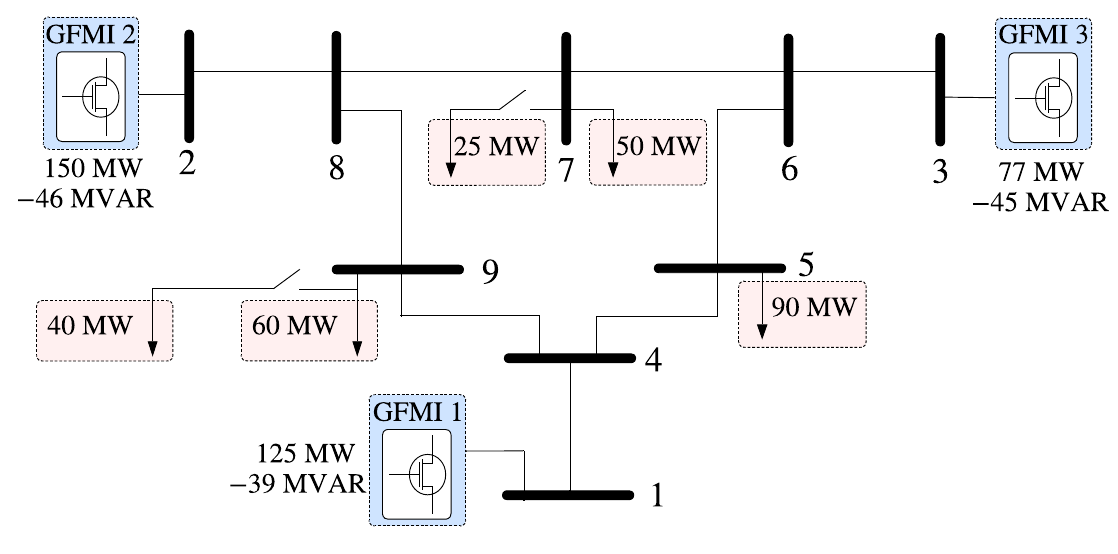} 
      \caption{Modified IEEE $230$-kV $9$-bus system.} \label{fig_nine_bus_system}
\end{figure}

Note that the EMT scenario captures the dynamics of voltage, current and frequency throughout the grid. To synthesize the frequency regulation scheme, we consider that the unknown state-space model is a multi-nodal frequency dynamics model that has $6$ states and $3$ input channels. This model reflects the swing dynamics at each GFMI bus and the DC power flow equation in the branches. The states measured from the EMT scenario are the frequencies at the GFMI buses, and the voltage phases at GFMI $2$ and $3$ with respect to the GFMI $1$. The input channels are the three GFMI's active power references. Moreover, the agents in this MAS are the three GFMIs, i.e., $|\mathcal{C}| = 3$. By defining these agents, the unknown frequency dynamics model becomes a MAS \cite{paulpatricia_nash_freq}.

\begin{figure}[t]
      \centering
      \includegraphics[width=\linewidth]{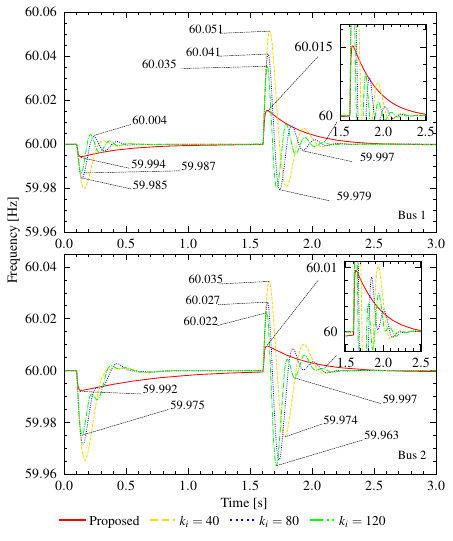} 
      \caption{Time-domain performances of  frequency regulation that result from using the proposed control framework and integral action strategies for the grid in Fig. \ref{fig_nine_bus_system}.} 
      \label{fig_freq_grid}
\end{figure}

First, we excite the power reference input of each GFMI with random square signals in the EMT scenario. We have found that these state trajectories from the EMT simulation contain numerical noise, which can make the SDP \eqref{eq_lqr_sdp_obj}-\eqref{eq_lqr_sdp_t} infeasible or resulting in zero state-feedback matrices. These are issues reported in \cite{zeng2024noisesensitivity}. In a real life setting, state trajectories are also corrupted with noise from sensors.  To circumvent these issues, we propose to use system identification to learn a state-space model using the input and state trajectories from the EMT simulation. In this manner, exciting the identified model with random signals results in state trajectories without noise. Input and state trajectories are used in the Nash seeking iterative LQR (Algorithm \ref{algo1}). Using this algorithm, we obtain the sets $\mathcal{K}$ and $\mathcal{P}$. Then we consider that each GFMI has access to only the frequencies of the other GFMIs. Voltage phases are not accessible. This information allows to build $C_i$ $\forall i \in \mathcal{C}$. Additionally, we add an arbitrary perturbation matrix $D_i = (0.01) \mathbf{1}$ $\forall i \in \mathcal{C}$, where $\mathbf{1}$ is a matrix where all entries are equal to one. Then, $\mathcal{K}$, $\mathcal{P}$, $C_i$ and $D_i \ \forall i \in \mathcal{C}$ are used as inputs to the SDP \eqref{eq_sdp_obj_d}-\eqref{eq_atten_d}, which ultimately synthesizes the set $\mathcal{F}$ for the GFMIs. Note that $\mathcal{F}$ contains the set of static output feedback controllers that steers the frequency back to its nominal value in the presence of load perturbations. Finally, we implement the controllers in the EMT scenario with load perturbations. 

Figure \ref{fig_freq_grid} compares the time-domain performances of the synthesized $\mathcal{F}$ and other strategies. These strategies are integral controllers with gain $k_i$ that receive the frequency deviation and act on the power reference input of each GFMI. As Fig. \ref{fig_freq_grid} and Table \ref{table_performance} show, the data-driven control scheme notably outperforms the other control schemes. It can attain up to $78\%$ less overshoot and up to $16\%$ less setting time.

\subsection{Voltage regulation}
\label{sec_voltage_reg}
This case study develops a control scheme to regulate voltage in a 20-kV two-bus system. We have designed a EMT scenario in Simulink that consists of a grid following inverter (GFLI), a GFMI, a long transmission line, and loads, as it is shown in Fig. \ref{fig_two_bus_system}. The GFLI is composed by a $dq-$based frame proportional-integral (PI) current controller, a L filter, a PWM, and an average voltage source converter. The GFMI is composed by a single-loop voltage controller, L filter, PWM, and average source converter. The GFMI generates a voltage at a fixed frequency (60 Hz), and shares its frequency signal with the GFLI. As in Section \ref{sec_freq_reg}, the transmission line is modeled as a $\pi$ structure.
\begin{figure}[t]
      \centering      \includegraphics[width=\linewidth]{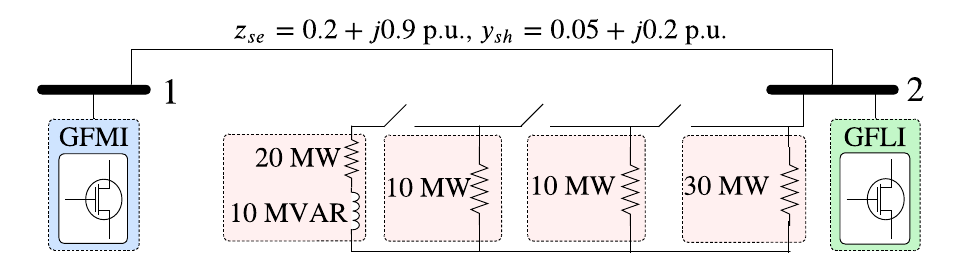} 
      \caption{$20$-kV two-bus system. The parameters $z_{se}$ and $y_{sh}$ are the series impedance and one of two shunt impedances, respectively, of the $\pi-$structure of the line. The base power is $100$ MVA.}
      \label{fig_two_bus_system}
\end{figure}

Consider that the unknown model is the $dq-$based frame small-signal model of the system, which has $13$ states and $3$ input channels. The agents in this system are the GFLI and the GFMI. This model reflects the EMT dynamics of all the grid components, i.e., GFMI, GFLI, transmission line and loads. The GFLI's states are: the components $d$ and $q$ of the state in the integral controller of the current controller, and the components $d$ and $q$ of L filter current. The GFMI's states are: the integral controller in the voltage magnitude controller, and the components $d$ and $q$ of the L filter current. Additionally, we consider the components $d$ and $q$ of the voltage of each bus, and the components $d$ and $q$ of the current in the RL branch of the transmission line. The input channels are: the GFMI's voltage reference, and GFLI's $d$ and $q$ components of the reference current. The process to determine the controllers for the agents is the same as the one described in Section \ref{sec_freq_reg}, except that this case uses Algorithm \ref{algo2} instead of Algorithm \ref{algo1}.

\begin{figure}[t]
      \centering
      \includegraphics[width=\linewidth]{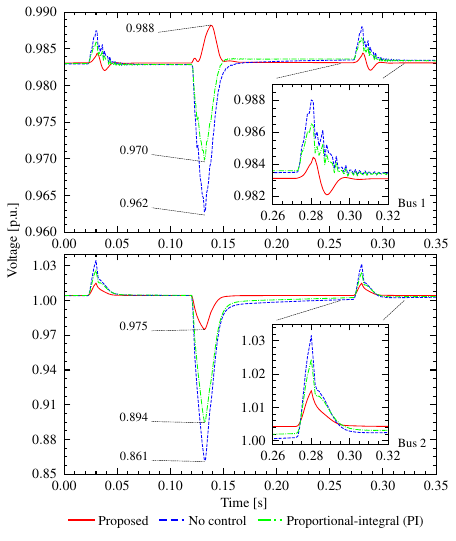} 
      \caption{Time-domain performances of voltage regulation that result from using the proposed control framework, no control, and PI control for the grid in Fig. \ref{fig_two_bus_system}. } 
      \label{fig_volt_grid}
\end{figure}

As Fig. \ref{fig_volt_grid} shows, the performance of the proposed control scheme is compared to a system with no control and a system with PI control. Note that the PI controller is configured through trial and error to obtain a close-to-zero steady state error and no instability.  As it is shown in Fig. \ref{fig_volt_grid} and Table \ref{table_performance}, the proposed control scheme can attain up to $74 \%$ less overshoot and up to $30 \%$ less settling time than the other strategies. 

\section{Challenges for data-driven framework implementation}
Subsections~\ref{sec_freq_reg} and \ref{sec_voltage_reg} illustrate the implementation of the proposed data-driven framework in realistic MAS. Through this implementation, we learn that:
\begin{enumerate}
    \item When exciting the actual nonlinear MAS to obtain data samples of state trajectories, they inevitably carry some amount of noise and numerical error. For instance, a state trajectory can report values in the range of $10^{-5}$ to $10^{-6}$, which can be effectively considered as zero values. However, when Algorithms~\ref{algo1} or \ref{algo2} use these corrupted data samples, they yield zero-valued matrices $\mathcal{P}$. As a result, the synthesis of output-feedback strategies fails. To overcome this issue, we may use noise-resilient SDPs, similar to \cite{depersis_lqr_datadriven, depersis_robust_lqr, dai_output_datadriven}, to obtain the NE using noisy data directly. However, even such noise-resilient SDPs can indeed compute zero-valued controllers, as mentioned  
    in \cite{zeng2024noisesensitivity}.
    \item An alternative choice to noise-resilient SDPs is to clean the data samples using a separate method from the proposed data-driven framework. Developing such method is beyond the scope of this work, instead, we use system identification in the case studies in Subsections~\ref{sec_freq_reg} and \ref{sec_voltage_reg}. Under this approach, we identify a non-linear or linear model from the realistic MAS. Off-the-shelf system identification tools are robust to noisy data. Next, we can run either Algorithm ~\ref{algo1} or \ref{algo2} using data samples generated from the identified MAS model. Next we input the NE solution into the SDP in \eqref{eq_sdp_obj_d}--\eqref{eq_atten_d} to ultimately synthesize the output-feedback strategies. We reiterate that other approaches, instead of system identification, can be used to generate clean data samples of state trajectories.
    \item Algorithm \ref{algo2} can be faster and less computationally expensive than Algorithm~\ref{algo1}. This is because Algorithm \ref{algo1} requires exciting the system and obtaining data samples in each iteration. In contrast, Algorithm \ref{algo2} requires data samples only once. 
    \item It is advisable to consider that the actual MAS is too complex; instead, a low-order version of the MAS can be better suited for the data-driven framework. As we experience in our internal tests, the proposed data-driven framework can fail if we aim to capture the dynamics of all the states and consider all input channels of the MAS. In turn, we can assume that the data-driven framework is applied to a reduced version of the MAS. For instance, in Subsection~\ref{sec_freq_reg}, the actual power system has 94 states, including states of voltage, current, and electrical frequency, and 9 input channels in the generators. However, as the target is to regulate the electrical frequency, we assume that the unknown model for data-driven control is a model of frequency dynamics. Using this knowledge, we can estimate that the unknown MAS intended to be used for data-driven control has only 6 states and 3 input channels. Next, by implementing the synthesized control in the high-fidelity non-linear model of the MAS, we can experimentally validate our control scheme. 
\end{enumerate}

\section{Conclusions}
Through this work, we have introduced a novel framework that uses offline data samples of input and state trajectories to ultimately synthesize output feedback strategies for LTI MAS. The framework is a two-step process. The first step consists of determining a NE solution. To this end, we have proposed two alternative algorithms: a Nash seeking iterative LQR and a Nash seeking Newton-Raphson method. The first algorithm uses only data samples of the input and state trajectories, without requiring to identify the matrices $A$ and $B_i \ \forall i \in \mathcal{C}$. In contrast, the second algorithm requires the identification of $B_i \ \forall i \in \mathcal{C}$; however, it does not require to have $\begin{bmatrix}
    A & B_i
\end{bmatrix}$ stabilizable for every $i \in \mathcal{C}$ while the first algorithm does. The second step of the framework consists of using the NE solution in a SDP, which has been built upon our prior work\cite{nash_outputfeedback}. Virtues of our framework include weaker controllability assumptions, the absence of online information, and data-driven estimation of states. We use a high fidelity non-linear MAS in power systems (EMT) to test our framework. Results show successful grid stabilization and provision of regulation services from the inverters in the grid using our proposed framework. It can attain up to 78\% less overshoot and 30\% faster settling time than traditional strategies such as proportional-integral control.

Finally, we discuss challenges encountered while implementing our proposed data-driven framework in realistic settings. Prior work usually tests their frameworks in numerical small-scale examples where the state-space model is known. However, these examples can overlook the barriers to implement the proposed data-driven framework in real-world MAS.

\section*{References}

\bibliography{bibliography.bib}

\begin{IEEEbiography}[{\includegraphics[width=1in,height=1.25in,clip,keepaspectratio]{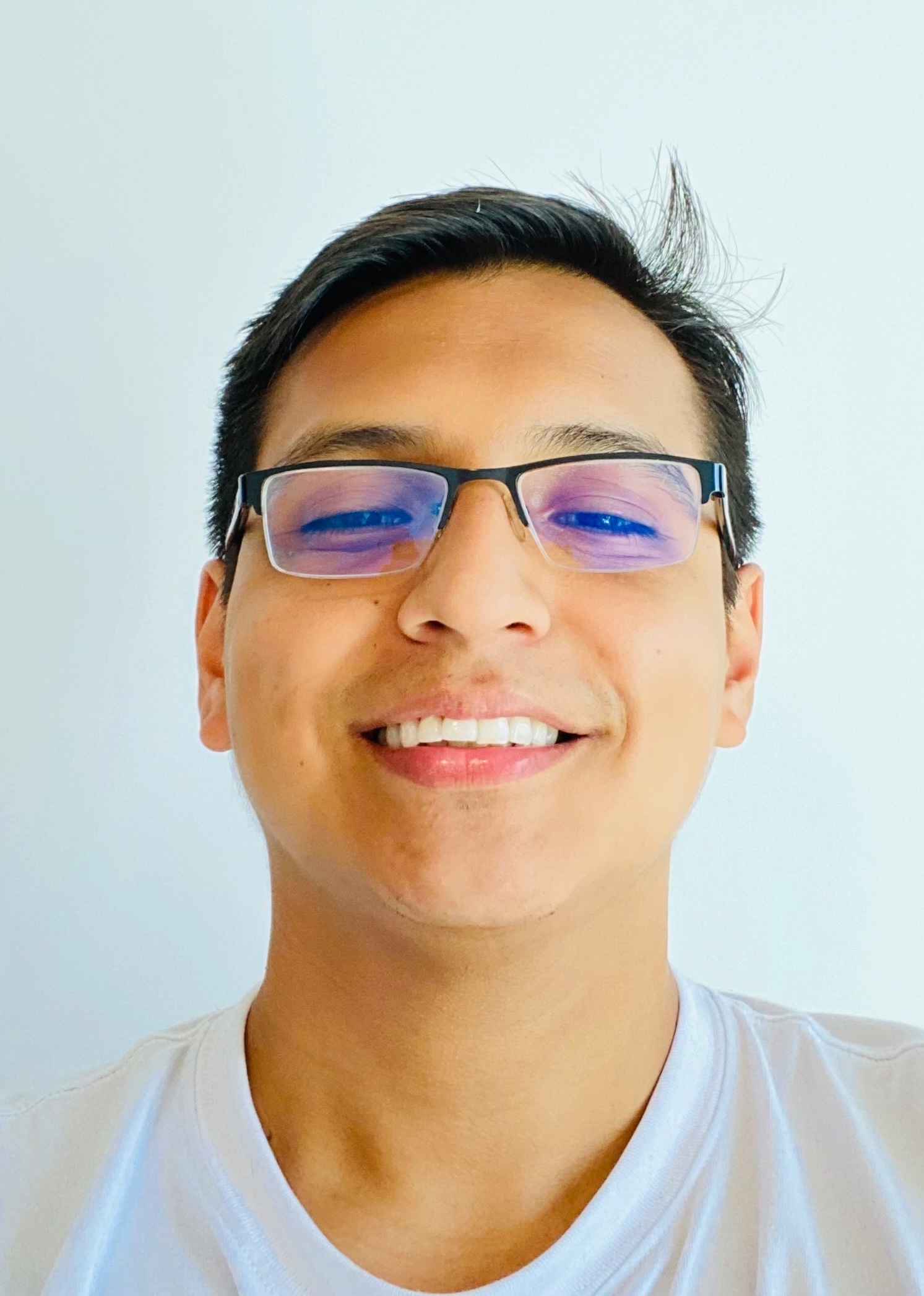}}]{Paul Serna-Torre}
received the B.Sc. degree in electrical engineering from Universidad Nacional de Ingenieria, Peru and the M.A. degree in Applied Mathematics from University of California San Diego, USA. He is currently working towards the Ph.D. degree at University of California San Diego, USA, with the supervision of professor Patricia Hidalgo-Gonzalez. His research interests include control of power system dynamics, and capacity expansion planning.
\end{IEEEbiography}

\begin{IEEEbiography}[{\includegraphics[width=1in,height=1.25in,clip,keepaspectratio]{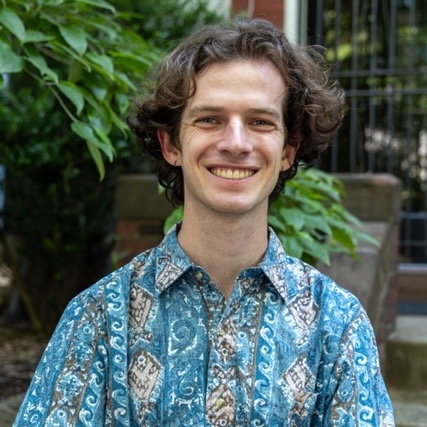}}]{Adam Sedlak}
received his B.A. degree in data science from University of California, Berkeley. He is currently working towards a Ph.D. degree in Mechanical and Aerospace Engineering at University of California San Diego, under the supervision of professor Patricia Hidalgo-Gonzalez. His research interests include power system dynamics and control, model reduction, as well as energy system optimization models.
\end{IEEEbiography}

\begin{IEEEbiography}[{\includegraphics[width=1in,height=1.25in,clip,keepaspectratio]{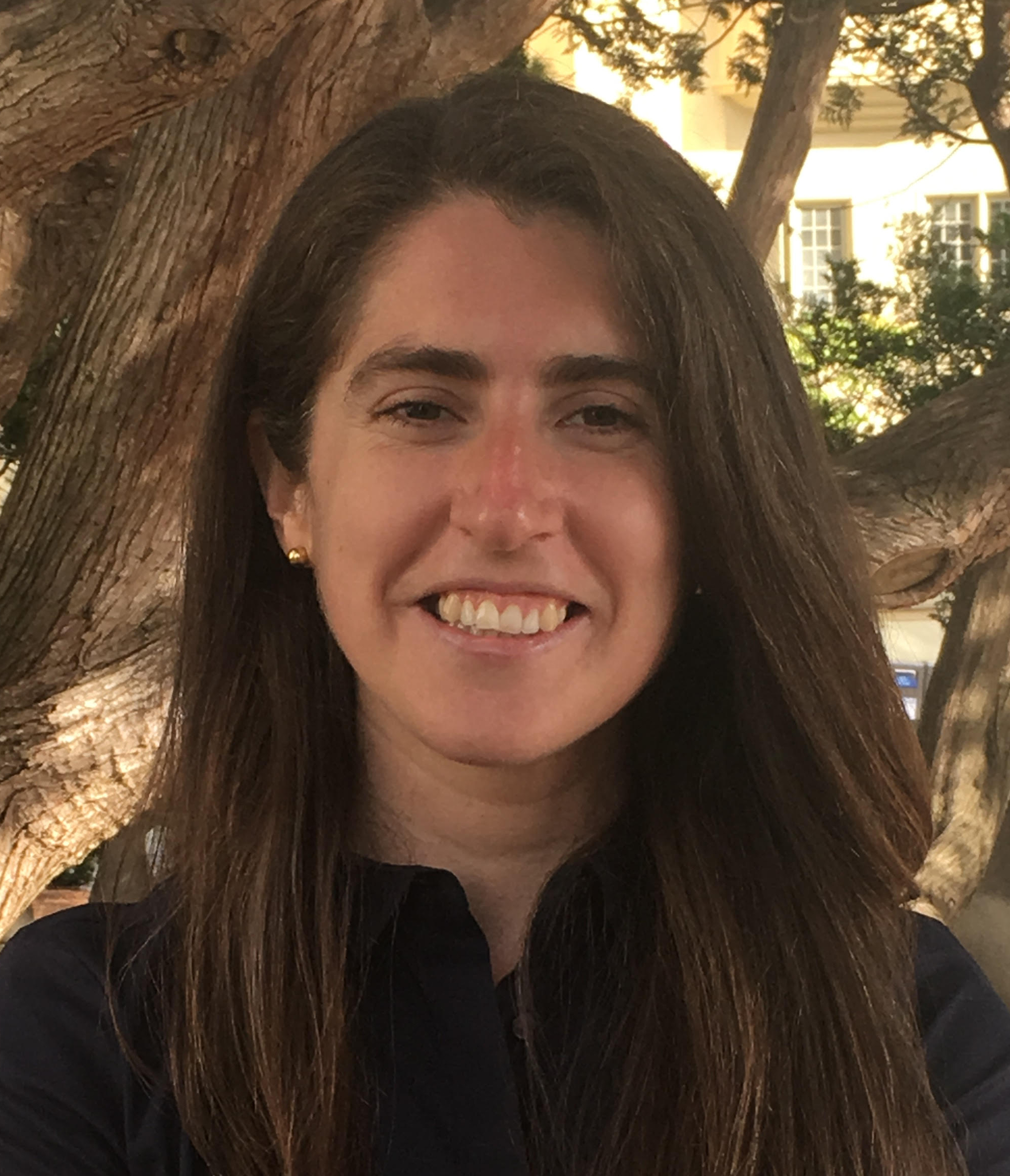}}]{Patricia Hidalgo-Gonzalez}
(Member, IEEE) received the B.Sc. degree in industrial and electrical engineering from Pontificia Universidad Catolica of Chile, and two M.Sc. degrees from the University of California, Berkeley in Energy and Resources and Electrical Engineering and Computer Sciences. She is an associate professor at University of California San Diego. Her research work focuses high penetration of renewable energy using optimization, control theory and machine learning. Prof. Hidalgo-Gonzalez is the 2025 Institute of Industrial and Systems Engineers Energy Systems Division Outstanding Young Investigator Awardee, an NSF GRFP fellow, Siebel Scholar in Energy, Rising Star in Electrical Engineering and Computer Science, received the Outstanding Graduate Student Instructor Award at UC Berkeley (for teaching Convex Optimization), has been awarded Best paper at the Power Systems Computation Conference 2020, and is part of the Editors’ Highlights in Nature Communications (best 50 papers recently published in “Engineering and Infrastructure”), her work was cited in The White House 2024 Economic Report of the President, among other recognitions. She is part of the IEEE Task forces ''Data-driven controls for distributed systems'' and ''Wildfire Aware-Risk Management (WARM) for the Electric Power Grid'' and has served as best paper judge in the IEEE Control Systems Society Technical Committee on Energy Systems. She has also served as the opening talk for the energy storage session in EPRI's 43rd Seminar on Resource Planning in Washington D.C., and as a speaker at a symposium from the U.S. National Academy of Sciences and the German National Academy of Sciences Leopoldina in Halle, Germany, among others.
\end{IEEEbiography}

\end{document}